%% file: main.tex
\documentclass[a4paper]{amsart}
\usepackage[utf8]{inputenc}
\input{header.tex}

\usepackage{url}
\usepackage[foot]{amsaddr}
\usepackage{geometry}
\usepackage{mathtools}

\author[U. S. Fjordholm]{Ulrik Skre Fjordholm}
\address{Department of Mathematics, University of Oslo, Postboks 1053 Blindern, 0316 Oslo, Norway}
\email{ulriksf@math.uio.no}
\author[K. O. Lye]{Kjetil Olsen Lye}
\address{Mathematics and Cybernetics, SINTEF, Pb. 124 Blindern, 0314 Oslo, Norway}
\email{kjetil.olsen.lye@sintef.no}

\title[Convergence rates with unbounded TV]{Convergence rates of monotone schemes for conservation laws for data with unbounded total variation}

\begin{document}
\maketitle
\begin{abstract}
We prove convergence rates of monotone schemes for conservation laws for H\"older continuous initial data with unbounded total variation, provided that the H\"older exponent of the initial data is greater than $\hf$. For strictly $\Lip^+$ stable monotone schemes, we prove convergence for any positive H\"older exponent. Numerical experiments are presented which verify the theory.
\end{abstract}

\section{Introduction}
Consider the scalar hyperbolic conservation law
\begin{equation}\label{eq:conservation_law}
\begin{split}
u_t+f(u)_x=0 \\ 
u(x,0)=u_0(x)
\end{split}
\end{equation}
where $f\in C^1(\R)$ is the flux function and $u_0\in L^1\cap L^\infty(\R)$ is the initial data. Equations of this form appear in a large number of applications, including scenarios where very irregular data is to be expected; we mention in particular flow in porous media \cite{Coclite2014,Krotkiewski2011} and turbulent flows (so-called ``Burgulence'') \cite{sinai,fris1}. While the study of qualitative properties of ``rough'' solutions of \eqref{eq:conservation_law} has been explored in detail (see e.g.~\cite{sinai,fris1}), the behavior of numerical methods for \eqref{eq:conservation_law} has received much less attention.

The purpose of this paper is to study the convergence rate of monotone numerical methods for \eqref{eq:conservation_law} in the presence of rough (say, piecewise H\"older continuous) initial data $u_0$. As is to be expected, the convergence rate deteriorates with lower regularity. We demonstrate in several numerical experiments that our estimates are sharp, or close to being sharp.

\subsection{Weak solutions of hyperbolic conservation laws}
As is well-known, solutions of nonlinear hyperbolic equations \eqref{eq:conservation_law} can develop shocks in finite time, making it necessary to interpret the equation in a weak manner. A weak solution of \eqref{eq:conservation_law} is a function $u\in L^1\cap L^{\infty}(\R\times \R_+)$ satisfying
\begin{equation}
\int_{\R_+}\int_{\R}u(x,t)\phi_x(x,t)+f(u(x,t))\phi(x,t)\,dx\,dt+\int_{\R}\phi(x,0)u_0(x)\,dx=0
\end{equation}
for all test functions $\phi\in C_c^\infty(\R\times \R_+)$. It is well-known that weak solutions are non-unique, so one introduces \emph{entropy conditions} to single out the physically relevant solutions. Concretely, we say that $u\in L^1\cap L^{\infty}(\R\times\R_+)$ is an entropy solution of \eqref{eq:conservation_law}, if for every pair of functions $\eta, q\colon\R\to\R$ where $\eta$ is convex and $q'=\eta'f'$, it holds that
\[\eta(u)_t+q(u)_x\leq 0\]
in the sense of distributions. In particular, it is sufficient to impose the entropy condition with respect to the Kruzkov entropy pairs, given by
\[
\eta(u,k)=|u-k|, \qquad q(u,k)=\sign(u-k)(f(u)-f(k)),\qquad u\in \R
\]
for all $k\in\R$. It was shown by Kruzkov (see e.g.~\cite{Dafermos} or \cite[Proposition 2.10]{front_tracking_risebro}) that entropy solutions of \eqref{eq:conservation_law} are unique.

\subsection{Finite volume methods for conservation laws}\label{sec:fvm}
This section briefly describes the conventional approach of numerical approximation of conservation laws through finite volume and finite difference methods. For a complete review, one can consult e.g.~\cite{leveque_green}.

We discretize the spatial domain $\R$ by partitioning it into a collection of cells $\cell_i \coloneqq [x_{i-\hf}, x_{i+\hf})\subset \D$
with corresponding cell midpoints $x_{i}\coloneqq\frac{x_{i+\hf}+x_{i-\hf}}{2}.$ For simplicity we assume that our mesh is equidistant, that is,
\[x_{i+\hf} - x_{i-\hf} \equiv \Dx \qquad \forall\ i\in\Z\]
for some $\Dx>0$. We discretize time by equidistant points, that is, we choose $t^n = n\Dt$ for $n\in\N_0$ for some $\Dt>0$.

For each cell $\cell_i$ and each point in time $t^n$ we let $v^{n}_i$ be an approximation of the cell average of $u$ at time $t^n$, $u_i^n \approx \intavg_{\cell_i}u(x,t^n)\dd x$ (here, $\intavg_C\coloneqq\frac{1}{|C|}\int_{C}$, where $|C|$ is the Lebesgue measure of a Lebesgue set $C\subset\R$). This approximation is computed according to the finite volume scheme
\begin{equation}\label{eq:fvs}
\begin{split}
\frac{v^{n+1}_i-v^n_i}{\Dt}+ \frac{F(v^n_{i}, v^n_{i+1})-F(v^n_{i-1}, v^n_{i})}{\Dx} = 0  \\
v^{0}_{i} = \intavg_{\cell_i} u_0(x)\dd x
\end{split}
\end{equation}
where $F$ is a \emph{numerical flux function}. We furthermore assume the numerical flux function is consistent with $f$ and locally Lipschitz continuous; more precisely, for every bounded set $K\subset \R$, there exists a constant $C_F>0$ such that
\begin{equation}\label{eq:FLipschitz}
\big|F(a,b)-f(a)\big|+\big|F(a,b)-f(b)\big|\leq C_F\big|b-a\big| \qquad \forall\ a,b\in K.
\end{equation}
We will frequently abuse notation and view grid functions $v\in \ell^1(\Z)$ as an element of $L^1(\R)$ under the inclusion $\ell^1(\Z)\hookrightarrow L^1(\R)$ which maps $v\mapsto \sum_i v_i \ind_{\cell_i}$.

\section{A modified Kuznetsov lemma}
Kuznetsov's lemma \cite{KUZNETSOV1976105} provides an explicit estimate of the difference between two (approximate) solutions of \eqref{eq:conservation_law} in terms of their relative (Kruzkov) entropy. In this section we recall Kuznetsov's lemma and prove a corollary which --- as opposed to Kuznetsov's original application of the lemma --- does not depend on $\TV(u_0)$ being bounded.

Fix now some final time $T>0$. Kuznetsov's lemma estimates approximation errors in the space
\[
\KuznetsovSpace\coloneqq\left\{u\colon\R_+\to \Lone \ \big|\ u\text{ and has right and left limits at all $(t,x)\in\R_+\times\R$}\right\}.
\]
For $u\in\KuznetsovSpace$ and $\sigma>0$ we define the moduli of continuity 
\[
\nu_t(u,\sigma)=\sup\left\{\|u(t+\tau)-u(t)\|_{\Lone}\mid 0<\tau\leq \sigma\right\}, \qquad
\nu(u,\sigma)=\sup_{t\in[0,T]}\nu_t(u, \sigma).
\]
Let $\omega\in C_c^\infty(\D)$ be a standard mollifier, i.e.~an even function satisfying $\supp \omega\subset [-1,1]$, $0\leq\omega\leq 1$ and $\int_{\D}\omega\dd x = 1.$ For $\epsilon>0$ we define $\omega_\epsilon(x)=\frac{1}{\epsilon}\omega(\frac{x}{\epsilon})$. For $\epsilon,\epsilon_0>0$, define
\[\Omega(x,x',s,s')=\omega_{\epsilon_0}(s-s')\omega_{\epsilon}(x-x')\qquad (x,x',s,s')\in\D^4.\]
For $\phi\in C_c^\infty(\D\times \R_+,\Ph)$, $k\in\R$ and $u,v\in\KuznetsovSpace$ we set
\begin{align*}
\Lambda_T(u,\phi, k)=&\int_0^T\int_{\D}\big(|u-k|\phi_t+q(u,k)\phi_x\big)\dd x \dd t \\
& -\int_{\D}|u(x,T)-k|\phi(x,T)\dd x+\int_{\D}|u(x,0)-k|\phi(x,0)\dd x,\\
\Lambda_{\epsilon,\epsilon_0}(u,v)=&\int_0^T\int_{\D}\Lambda_T(u, \Omega(\cdot, x', \cdot, s'), v(x',s'))\dd x'\dd s'.
\end{align*}

\begin{lemma}[Kuznetsov's lemma \cite{KUZNETSOV1976105}]\label{lemma:kuznetsov}
Let $v\in\KuznetsovSpace$ and let $w$ be an entropy solution of \eqref{eq:conservation_law}. If $0<\epsilon_0<T$ and $\epsilon>0$, then
\begin{align*}
\|v(\cdot,T)-w(\cdot,T)\|_{\Lone} &\leq \|v_0-w_0\|_{\Lone} +\TV(w_0)\big(2\epsilon + \epsilon_0\|f\|_{\Lip}\big) \\ &\quad+\nu(v,\epsilon_0)-\Lambda_{\epsilon, \epsilon_0}(v,w)
\end{align*}
where $v_0=v(\cdot, 0)$ and $w_0=w(\cdot,0)$.
\end{lemma}
The following is a straightforward extension of~\cite[Lemma 4 and Theorem 4]{KUZNETSOV1976105}.

\begin{lemma}\label{lem:generalconv}
Let $u_0\in L^1(\D)\cap L^\infty(\D)$ and let $v^\Dx$ be the solution computed by a monotone finite volume scheme \eqref{eq:fvs} with initial data $v^\Dx_0$. Then
\begin{equation}\label{eq:kuznetsovestimate}
\begin{split}
\|u(T)-v^{\Dx}(T)\|_{L^1(\R)} &\leq 2\|u_0-v^\Dx_0\|_{\Lone} + \TV(v^\Dx_0)\big(2\epsilon + \epsilon_0\|f\|_{\Lip} + 2C_F\max(\epsilon_0,\Dt)\big) \\
&\quad +  C\left(\frac{C_F\Dx}{\epsilon} + \frac{\|f\|_{\Lip}\Dt}{\epsilon_0}\right)\sum_{n=0}^N\TV(v^\Dx(t^n))\Dt.
\end{split}
\end{equation}
for any $T>0$, $\epsilon>0$ and $0<\epsilon_0<T$, for some $C>0$ only depending on the choice of smoothing kernel $\omega$.
\end{lemma}
\begin{proof}
	Let $w$ be the entropy solution of \eqref{eq:conservation_law} with $w_0=v^\Dx_0$. Then 
	\begin{align*}
	\|u(T)-v^\Dx(T)\|_{\Lone} &\leq \|u(T)-w(T)\|_{\Lone} + \|w(T)-v^\Dx(T)\|_{\Lone} \\
	&\leq \|u_0-v^\Dx_0\|_{\Lone} + \|w(T)-v^\Dx(T)\|_{\Lone}
	\end{align*}
	by the stability of entropy solutions in $\Lone$ (see e.g.~\cite[Theorem 1]{Kru70} or \cite[Proposition 2.10]{front_tracking_risebro}). We estimate the second term using \Cref{lemma:kuznetsov}. For notational convenience, denote
    \[\eta_i^n=|v^n_i-k| \qquad \text{and} \qquad q^n_i=q(v^n_i,k).\]
    Without loss of generality we may assume that $T=t^{N+1}$ for some $N\in\N$. Then
\begin{align*}
-\Lambda_T(v^\Dx,\phi, k) &=-\sum_{n=0}^N\sum_{i=-\infty}^\infty\int_{t^n}^{t^{n+1}}\int_{x_{i-\hf}}^{x_{i+\hf}}(\eta^n_i\phi_t+q^n_i\phi_x)\dd x \dd t \\
&\quad+\sum_{i=-\infty}^\infty\int_{x_{i-\hf}}^{x_{i+\hf}}\eta^{N+1}_i\phi(x,T)\dd x - \sum_{i=-\infty}^\infty\int_{x_{i-\hf}}^{x_{i+\hf}}\eta^0_i\phi(x,0)\dd x\\
&= -\sum_{n=0}^N\sum_{i=-\infty}^\infty\int_{x_{i-\hf}}^{x_{i+\hf}}\eta^n_i\left(\phi(x,t^{n+1})-\phi(x,t^{n})\right)\dd x\\
    &\quad+\sum_{n=0}^N\sum_{i=-\infty}^\infty\int_{t^n}^{t^{n+1}}q^n_i\left(\phi(x_{i+\hf},t)-\phi(x_{i-\hf},t)\right) \dd t \\
    &\quad+\sum_{i=-\infty}^\infty\int_{x_{i-\hf}}^{x_{i+\hf}}\eta^{N+1}_i\phi(x,T)\dd x - \sum_{i=-\infty}^\infty\int_{x_{i-\hf}}^{x_{i+\hf}}\eta^0_i\phi(x,0)\dd x \\
\intertext{(summation by parts)}
    &= \sum_{n=0}^N\sum_{i=-\infty}^\infty\int_{x_{i-\hf}}^{x_{i+\hf}}\left(\eta^{n+1}_i-\eta^{n}_i\right)\phi(x,t^{n+1})\dd x\\
    &\quad+\Dx\sum_{n=0}^N\sum_{i=-\infty}^\infty\int_{t^n}^{t^{n+1}}\left(q^n_{i+1}-q^n_i\right)\phi(x_{i+\hf},t)\dd t \\
\intertext{(set $\bar{\phi}_i^n\coloneqq\frac{1}{\Dx}\int_{x_{i-\hf}}^{x_{i+\hf}}\phi(x,t^{n+1})\dd x$ and $\bar{\phi}_\iphf^\nphf \coloneqq \frac{1}{\Dt}\int_{t^n}^{t^{n+1}}\phi(x_{i+\hf},t)\dd t$)}
    &= \Dx\sum_{n=0}^N\sum_{i=-\infty}^\infty\left(\eta^{n+1}_i-\eta^{n}_i\right)\bar{\phi}_i^n +\Dt\sum_{n=0}^N\sum_{i=-\infty}^\infty\left(q^n_{i+1}-q^n_i\right)\bar{\phi}_\iphf^\nphf.
\end{align*}
Let $Q^n_\iphf\coloneqq F\big(v^n_i\vee k, v^n_{i+1}\vee k\big) - F\big(v^n_i\wedge k, v^n_{i+1}\wedge k\big)$ be the Crandall--Majda numerical entropy flux, so that
\[\eta^{n+1}_i-\eta^n_i+\frac{\Dt}{\Dx} \big(Q^n_\iphf-Q^n_\imhf\big)\leq 0\]
(see e.g.~\cite{CM80} or \cite[(3.33)]{front_tracking_risebro}). It is not hard to show that $Q$ is Lipschitz continuous,
\begin{equation}\label{eq:QLipschitz}
|Q_\iphf^n-q_i^n| \leq 2C_F|v_{i+1}^n-v_i^n|,
\end{equation}
where $C_F$ is the Lipschitz constant for $F$ (cf.~\eqref{eq:FLipschitz}). Assuming now that $\phi$ is non-negative, we obtain from the above discrete entropy inequality
\begin{align*}
    -\Lambda_T(v^\Dx,\phi, k)&\leq -\Dt\sum_{n=0}^N\sum_{i=-\infty}^\infty\big(Q^n_\iphf-Q^n_\imhf\big)\bar{\phi}_i^{n+1} +\Dt\sum_{n=0}^N\sum_{i=-\infty}^\infty\left(q^n_{i+1}-q^n_i\right)\bar{\phi}_\iphf^\nphf \\
    &= \Dt\sum_{n=0}^N\sum_{i=-\infty}^\infty\big( Q^n_\iphf-q^n_i\big)\big(\bar{\phi}_{i+1}^{n+1}-\bar{\phi}_i^{n+1}\big)\\
    &\quad+\Dt\sum_{n=0}^N\sum_{i=-\infty}^\infty\left(q^n_{i+1}-q^n_{i}\right)\big(\bar{\phi}_\iphf^\nphf - \bar{\phi}_{i+1}^{n+1}\big) \\
\intertext{(using \eqref{eq:QLipschitz} and the Lipschitz continuity $\|q\|_{\Lip}\leq\|f\|_{\Lip}$)}
    &\leq \Dt\sum_{n=0}^N\sum_{i=-\infty}^\infty |v_{i+1}^n-v_i^n|\Big(C_F\big|\bar{\phi}_{i+1}^{n+1}-\bar{\phi}_i^{n+1}\big| + \|f\|_{\Lip}\big|\bar{\phi}_\iphf^\nphf - \bar{\phi}_{i+1}^{n+1}\big|\Big) \\
\intertext{(smoothness of $\phi$)}
	&\leq \Dt\big(C_F\Dx\|\partial_x\phi\|_{L^\infty} + \|f\|_{\Lip}\Dt\|\partial_t\phi\|_{L^\infty}\big) \sum_{n=0}^N\sum_{i=-\infty}^\infty |v_{i+1}^n-v_i^n|.
\end{align*}
From this estimate we obtain
\begin{align*}
&-\Lambda_{\epsilon_0, \epsilon}(v^\Dx,w) = -\int_0^T\int_{\R} \Lambda_T(v^\Dx, \omega_{\epsilon_0}(\cdot-s)\omega_{\epsilon}(\cdot-y), w)\dd y \dd s \\
&\leq \Dt\int_0^T\int_{\R}\big(C_F\Dx\|\omega_{\epsilon_0}(\cdot-s)\|_{L^\infty}\|\omega_{\epsilon}'(\cdot-y)\|_{L^\infty} + \|f\|_{\Lip}\Dt\|\omega_{\epsilon_0}'(\cdot-s)\|_{L^\infty}\|\omega_{\epsilon}(\cdot-y)\|_{L^\infty}\big) \\
&\mathop{\hphantom{\leq \Dt\int_0^T\int_{\R}}} \times\sum_{n=0}^N\sum_{i=-\infty}^\infty |v_{i+1}^n-v_i^n| \dd y \dd s \\
&\leq C\Dt\left(\frac{C_F\Dx}{\epsilon} + \frac{\|f\|_{\Lip}\Dt}{\epsilon_0}\right)\sum_{n=0}^N\sum_{i=-\infty}^\infty |v_{i+1}^n-v_i^n|
\end{align*}
for some constant $C>0$ only depending on $\omega$. 

It remains to estimate $\nu(v^\Dx,\epsilon_0)$. The standard estimate
\[
|v_i^{n+1}-v_i^n| \leq \frac{\Dt}{\Dx}C_F\big(|v_{i+1}^n-v_i^n| + |v_i^n-v_{i-1}^n|\big)
\]
yields
\begin{align*}
\nu(v^\Dx,\epsilon_0) &\leq \Dx\sum_i \max(\epsilon_0,\Dt)\frac{1}{\Dx}C_F\big(|v_{i+1}^n-v_i^n| + |v_i^n-v_{i-1}^n|\big) = 2C_F\max(\epsilon_0,\Dt)\TV(v^n).
\end{align*}
\end{proof}

\section{Convergence rates for irregular data}\label{sec:convergence_rates}

With the Kuznetsov lemma and its corollary in place, we are now in place to prove convergence rates for \eqref{eq:fvs} with irregular data. We start with some preliminaries in Section \ref{sec:prelim} before proving convergence rates in Section \ref{sec:convergence_rates}.

\subsection{Preliminaries}\label{sec:prelim}
We define the \emph{discrete Lip$^+$ (semi-)norm} as the sublinear functional
\[
\DLipNorm{v}\coloneqq\sup_{i\in\Z}\frac{v_{i+1}-v_i}{\Dx} \qquad \text{for } v\in\ell^\infty(\R).
\]
Following \cite{NT92} (see also \cite{FjoSol16}), we say that a numerical flux function is \emph{(strictly) Lip$^+$ stable} if 
\begin{equation}\label{eq:lipplusdef}
\DLipNorm{v^{n+1}}\leq \frac{1}{\DLipNorm{v^n}^{-1}+\beta\Dt}
\end{equation}
for some $\beta\geq0$ ($\beta>0$, respectively) which is independent of $\Dt,\Dx$. Iterating \eqref{eq:lipplusdef}, it holds in particular that
\[
\DLipNorm{v^{n}} \leq \frac{1}{\DLipNorm{v^0}^{-1}+\beta t^n} \qquad \forall\ n\in\N.
\]
It was shown in \cite{NT92} that the Lax--Friedrichs, Engquist--Osher and Godunov schemes are all strictly $\Lip^+$ stable. (The Roe scheme is non-strictly $\Lip^+$ stable.) The concept of $\Lip^+$ stability is motivated by the Oleinik entropy condition for conservation laws with strictly convex flux functions \cite{Ole57}, which states that the $\Lip^+$ seminorm $\LipNorm{u}\coloneqq\sup_{x\neq y} \frac{u(x)-u(y)}{x-y}$ of a solution of \eqref{eq:conservation_law} should decrease over time at a rate proportional to $t^{-1}$; more precisely,
\[
\LipNorm{u(t)} \leq \frac{1}{\LipNorm{u_0}^{-1} + \beta_0 t}
\]
where $0\leq\beta_0\leq f''(v)$ for all $v\in\R$.

For a function $g\in L^1(\R)$ we define its total variation as 
\[
\TV(g) = \sup\left\{\int_{\R} g(x)\phi'(x)\,dx\mid \phi\in C_c^1(\R), \|\phi\|_{L^\infty}\leq 1\right\}.
\]
We say that a finite volume scheme is \emph{total variation diminishing} (TVD) if for every $u_0\in \BV(\R)$, we have $\TV(v^{n+1})\leq \TV(v^{n})$ for all $n\geq 0$. We say that the scheme is \emph{monotone} if for all cell averaged initial data $u^0,v^0$ with $u^0_j\leq v^0_j$ for all $j\in \Z$, we have $u^n_{j}\leq u^n_{j}$ for all $n>0$ and $j\in\Z$.

\begin{lemma}\label{lem:approximation_holder}
Let $u\in C^\alpha_c(\R)$ for some $\alpha>0$ and let $u^{\Dx}\in \ell^1(\Z)$ be the volume averages of $u$,
\[u^{\Dx}_i=\intavg_{\cell_i}u(x)\,dx\qquad i\in\Z.\]
Then
\begin{subequations}
\begin{align}
\|u^{\Dx}-u\|_{L^1(\R)} &\leq C\Dx^\alpha,	\label{eq:l1projectionerror} \\
\TV(u^\Dx) &\leq C\Dx^{\alpha-1},	\label{eq:tvprojectionerror}
\end{align}
\end{subequations}
where $C$ only depends on $\alpha$ and the size of the support of $u$.
\end{lemma}
\begin{proof}
Let $K=\{i\in\Z\mid \cell_i\cap \supp u\neq \emptyset\}$. Then
\begin{align*}
	\|u^{\Dx}-u\|_{L^1(\R)}&=\sum_{i\in K}\int_{\cell_i}|u^{\Dx}(x)-u(x)|\,dx
	=\sum_{i\in K}\int_{\cell_i}\left|\intavg_{\cell_i}u(y)-u(x)\,dy\right|\,dx\\
	&\leq\sum_{i\in K}\int_{\cell_i}\intavg_{\cell_i}|u(y)-u(x)|\,dy\,dx
	\leq \|u\|_{C^\alpha}\sum_{i\in K}\int_{\cell_i}\intavg_{\cell_i}|x-y|^{\alpha}\,dy\,dx\\
	&\leq \|u\|_{C^\alpha}\sum_{i\in K}\int_{\cell_i}\Dx^{\alpha} dx
	\leq \|u\|_{C^\alpha}\sum_{i\in K}\Dx^{\alpha+1}\\
	&=C \Dx^{\alpha}
\end{align*}
where $C=\|u\|_{C^\alpha}\Dx|K|$ and $|K|$ is the Lebesgue measure of $K$. Similarly,
\begin{align*}
\TV(u^\Dx) = \sum_{i\in K}\left|\intavg_{\cell_i}u(x+\Dx)-u(x)\,dx\right| \leq \sum_{i\in K} \|u\|_{C^\alpha}\Dx^\alpha = C\Dx^{\alpha-1}
\end{align*}
for the same constant $C$ as above.
\end{proof}

\subsection{Convergence rates}\label{sec:convrate}
Without any assumptions on scheme beyond being monotone, we can only prove convergence rates for initial data whose H\"older exponent is not smaller than $1/2$, which the following theorem makes precise.
\begin{theorem}\label{thm:monotone_theorem} 
For a flux function $f\in C^1(\R)$, let $u\colon\R\times[0,T]\to\R$ be the entropy solution of \eqref{eq:conservation_law} with initial data $u_0\in C_c^\alpha(\R)$ for some $\alpha\in(0,1)$. Let $(v_i^n)_{i,n}$ be generated by a monotone finite volume scheme \eqref{eq:fvs} with initial data $u_0$. Then
\begin{equation}
\|u(T)-v^{\Dx}(T)\|_{L^1(\R)} \leq C\Dx^{\alpha-1/2}
\end{equation}
for any $T>0$, for some $C>0$ only depending on $f$ and $u_0$.
\end{theorem}
\begin{proof}
Since $u_0$ is H\"older continuous with exponent $\alpha$, \Cref{lem:approximation_holder} implies that $\TV(v^\Dx_0)\leq C\Dx^{\alpha-1}<\infty$, and since the scheme is TVD we get $\TV(v^{\Dx}(t^n))\leq  \TV(v^\Dx_0)$. Hence,
\[
\sum_{n=0}^N\TV(v^\Dx(t^n))\Dt\leq CT\Dx^{\alpha-1}.
\]
We note furthermore that
\[\|u_0-v^\Dx_0\|_{\Lone}\leq C\Delta x^{\alpha}.\]
Combining the above estimates with that of Lemma \ref{lem:generalconv}, we see that
\begin{equation}
\begin{split}
\|u(T)-v^{\Dx}(T)\|_{L^1(\R)} &\leq 2C\Delta x^{\alpha} + C\Dx^{\alpha-1}\big(2\epsilon + \epsilon_0\|f\|_{\Lip} + 2C_F\max(\epsilon_0,\Dt)\big) \\
&\quad +  CT\left(\frac{C_F\Dx}{\epsilon} + \frac{\|f\|_{\Lip}\Dt}{\epsilon_0}\right)\Dx^{\alpha-1}.
\end{split}
\end{equation}
Choosing $\epsilon=\epsilon_0=\Dx^{1/2}$ yields
\begin{equation}
\|u(T)-v^{\Dx}(T)\|_{L^1(\R)} \leq 2C_1\Delta x^{\alpha} + C_2\Dx^{\alpha-1/2} \leq C\Dx^{\alpha-1/2},
\end{equation}
which was what we wanted.
\end{proof}

We can improve the somewhat suboptimal convergence rate of $\Dx^{\alpha-1/2}$ for $\Lip^+$-stable schemes. To this end we need the following result.

\begin{lemma}\label{lem:lipplus}
Let $u_0\in L^1\cap L^\infty(\R)$ have compact support and let $u$ be the entropy solution of \eqref{eq:conservation_law}. Let $(v_i^n)_{i,n}$ be generated by a strictly $\Lip^+$ stable finite volume scheme \eqref{eq:fvs}. 
	Then 
	\[
	\sum_{n=0}^N \TV(v^n)\Dt \leq 
	C\left(\DLipNorm{v^0}\Dt + \frac{1}{\beta}\log \big(1+\beta t^N\DLipNorm{v^0}\big)\right)
	\]
	where $t^N\leq T$ and $C>0$ is independent of $\Dx$.
\end{lemma}
\begin{proof}
	Let $M>0$ be such that $\supp v^\Dx(t)\in[-M,M]$ for all $t\in[0,T]$ and let $I\in\N$ be such that $x_{I-1}<M\leq x_I$. The compact support of $v$ and the strict $\Lip^+$ stability imply that
	\begin{align*}
		\TV(v^n) &= \sum_{i=-I}^{I} |v_{i+1}^n-v_i^n| = 2\sum_{i=-I}^I \big(v_{i+1}^n-v_i^n\big)^+ \\
		&\leq \sum_{i=-I}^{I} \frac{1}{\DLipNorm{v^0}^{-1}+\beta t^n}\Dx \leq 2M\frac{1}{\DLipNorm{v^0}^{-1}+\beta t^n}.
	\end{align*}
	Hence,
	\begin{equation}
	\label{eq:tv_not_so_sharp}
	\begin{aligned}
		\sum_{n=0}^N \TV(v^n)\Dt &\leq 2M\sum_{n=0}^N\frac{1}{\DLipNorm{v^0}^{-1}+\beta t^n} \Dt \\
		&\leq 2M\left(\frac{\Dt}{\DLipNorm{v^0}^{-1}} + \frac{1}{\beta}\left(\log \big(\DLipNorm{v^0}^{-1}+\beta t^N\big) - \log\big(\DLipNorm{v^0}^{-1}\big)\right)\right) \\
		&= 2M\left(\DLipNorm{v^0}\Dt + \frac{1}{\beta}\log \big(1+\beta t^N\DLipNorm{v^0}\big)\right).\qedhere
		\end{aligned}
	\end{equation}
\end{proof}

\begin{theorem}	\label{thm:maintheorem}
	For a strictly convex flux function $f\in C^1(\R)$, let $u\colon\R\times[0,T]\to\R$ be the entropy solution of \eqref{eq:conservation_law} with initial data $u_0\in C_c^\alpha(\R)$ for some $\alpha\in(0,1)$. Let $(v_i^n)_{i,n}$ be generated by a strictly $\Lip^+$ stable, monotone finite volume scheme \eqref{eq:fvs} with initial data $u_0$. Then
	\begin{equation}\label{eq:optimalconvrate}
	\|u(T)-v^{\Dx}(T)\|_{L^1(\R)} \leq C_{L,M,f,\beta}\sqrt{\log\big(1+C_F\beta T\Dx^{\alpha-1}\big)}\Dx^{\alpha/2},
	\end{equation}
for any $T>0$, for some $C_{L,M,f,\beta}>0$. For small $\Dx>0$ this yields the ``almost $\Dx^{\alpha/2}$'' estimate
\begin{equation}\label{eq:optimalasymptoticconvrate}
	\|u(T)-v^{\Dx}(T)\|_{L^1(\R)} \leq C\Dx^{\alpha/2}\sqrt{-\log\Dx}.
\end{equation}
\end{theorem}
\begin{proof}
	\Cref{lem:lipplus} and \Cref{lem:approximation_holder} imply
	\[\sum_{n=0}^N \TV(v^n)\Dt \leq C\left(\|u_0\|_{C^\alpha}\Dx^\alpha + \frac{1}{\beta}\log \big(1+\|u_0\|_{C^\alpha}\beta t^N\Dx^{\alpha-1}\big)\right)\]
	for some $C>0$ independent of $\Dx$. Inserting this and the bounds $\|u_0-v^\Dx_0\|_{\Lone}\leq C\Dx^\alpha$ and $\TV(v^\Dx_0)\leq C\Dx^{\alpha-1}$ from \Cref{lem:approximation_holder} into the Kuznetsov estimate \eqref{eq:kuznetsovestimate}, we produce
	\begin{equation}	\label{eq:tvintegratedboundsinserted}
	\begin{split}
	\|u(T)-v^{\Dx}(T)\|_{L^1(\R)} &\leq C\Dx^\alpha + C\Dx^{\alpha-1}\big(\epsilon + \epsilon_0\|f\|_{\Lip} + \max(\epsilon_0,\Dt)\big) \\
	&\quad +  C\left(\frac{C_F\Dx}{\epsilon} + \frac{\|f\|_{\Lip}\Dt}{\epsilon_0}\right) \left(\Dx^\alpha + \frac{1}{\beta}\log\big(1+\|u_0\|_{C^\alpha}\beta T\Dx^{\alpha-1}\big)\right).
	\end{split}
	\end{equation}
	Defining
	\[
	Q(\Dx,\alpha)=\Dx^\alpha + \frac{1}{\beta}\log\big(1+\|u_0\|_{C^\alpha}\beta T\Dx^{\alpha-1}\big)
	\]
	and setting $\epsilon = \epsilon_0 = \Dx^{1-\alpha/2}\sqrt{Q(\Dx,\alpha)}$ yields
	\begin{equation}
	\begin{split}
	\|u(T)-v^{\Dx}(T)\|_{L^1(\R)} &\leq C\Dx^\alpha + C_{L,M,f}\sqrt{Q(\Dx,\alpha)}\Dx^{\alpha/2}
	\end{split}
	\end{equation}
	Since $\alpha\in(0,1)$, the second term on the right-hand side dominates, so we obtain \eqref{eq:optimalconvrate}. To get \eqref{eq:optimalasymptoticconvrate} we estimate
	\[
	\log\big(1+C\Dx^{\alpha-1}\big) \lesssim \log(C)+\log(\Dx^{\alpha-1}) \lesssim -(1-\alpha)\log(\Dx).
	\]
\end{proof}


\section{Numerical examples}
We consider three scalar conservation laws: Burgers' equation where $f(u)=u^2/2$, a cubic conservation law where $f(u)=u^3/3$, and lastly a linear conservation law where $f(u)=u$. The initial data will be given as fractional Brownian motion with varying Hurst exponent $H$. Introduced by Mandelbrot et al.~\cite{mandelbrot_fractional}, fractional Brownian motion can be seen as a generalization of standard Brownian motion with a scaling exponent different than $1/2$. We set
\[
u^H_0(\omega; x)\coloneqq B^H(\omega; x)\qquad \omega\in\Omega,\ x\in [0,1],
\]
where $B^H$ is fractional Brownian motion with Hurst exponent $H\in(0,1)$. Brownian motion corresponds to a Hurst exponent of $H=1/2$. 

To generate fractional Brownian motion, we use the random midpoint displacement method originally introduced by L\'evy~\cite{Levy1965} for Brownian motion, and later adapted for fractional Brownian motion~\cite{Fournier:1982:CRS:358523.358553,Voss1991}. Consider a uniform partition $0=x_{\hf}<\dots<x_{N+\hf}=1$ with $x_\iphf - x_\imhf \equiv \Dx$, where $N=2^k+1$ is the number of cells for some $k\in\N$. We first fix the endpoints
\[u^{H, \Delta x}_{1}(\omega; 0) =0\qquad u^{H, \Delta x}_{N}(\omega;0)=X_0(\omega), \]
where $(X_k)_{k\in\N}$ is a collection of normally distributed random variables with mean 0 and variance 1. Recursively, we set
\[u^{H, \Delta x}_{2^{k-l-1}(2j+1)}(\omega; 0)=\frac{1}{2}\left (u^{H, \Delta x}_{2^{k-l}(j+1)}(\omega; 0)+u^{H, \Delta x}_{2^{k-l}j}(\omega; 0)\right)+\sqrt{\frac{1-2^{2H-2}}{2^{2lH}}}X_{2^l+j}(\omega)\]
for $l=0,\ldots,k$ and for $j=0,\ldots,2^l$. That is, we bisect every interval and set the middle value to the average of the neighboring values plus some Gaussian random variable. See \Cref{fig:fractional_sample_burgers} (left column) for an example with $H=0.125$, $H=0.5$ and $H=0.75$. The initial data is normalized to have values in $[-1,1]$.

\begin{figure}
	\begin{subfigure}{0.3\textwidth}
		\InputImage{0.9\textwidth}{0.6\textwidth}{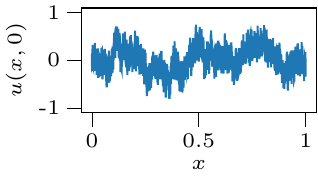}
		\subcaption{$H=0.125,$ $T=0$}
	\end{subfigure}
	\begin{subfigure}{0.3\textwidth}
		\InputImage{0.9\textwidth}{0.6\textwidth}{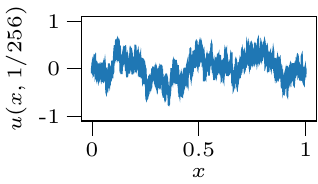}		
		\subcaption{$H=0.125,$ $T=1/256$}
	\end{subfigure}
	\begin{subfigure}{0.3\textwidth}
		\InputImage{0.9\textwidth}{0.6\textwidth}{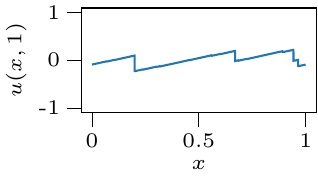}
		\subcaption{$H=0.125,$ $T=1$}
	\end{subfigure}

	\begin{subfigure}{0.3\textwidth}
	\InputImage{0.9\textwidth}{0.6\textwidth}{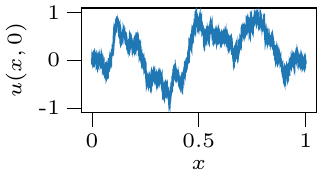}
	\subcaption{$H=0.5,$ $T=0$}
	\end{subfigure}
	\begin{subfigure}{0.3\textwidth}
		\InputImage{0.9\textwidth}{0.6\textwidth}{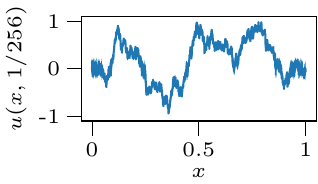}
		\subcaption{$H=0.5,$ $T=1/256$}
	\end{subfigure}
	\begin{subfigure}{0.3\textwidth}
		\InputImage{0.9\textwidth}{0.6\textwidth}{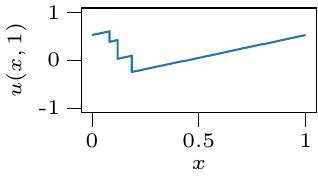}
		\subcaption{$H=0.5,$ $T=1$}
	\end{subfigure}

	\begin{subfigure}{0.3\textwidth}
		\InputImage{0.9\textwidth}{0.6\textwidth}{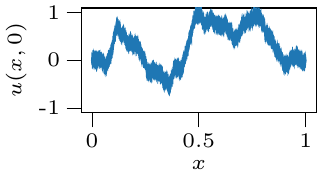}
		\subcaption{$H=0.75,$ $T=0$}
	\end{subfigure}
	\begin{subfigure}{0.3\textwidth}
		\InputImage{0.9\textwidth}{0.6\textwidth}{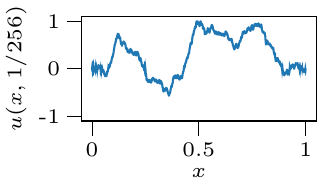}
		\subcaption{$H=0.75,$ $T=1/256$}
	\end{subfigure}
	\begin{subfigure}{0.3\textwidth}
		\InputImage{0.9\textwidth}{0.6\textwidth}{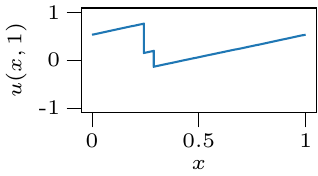}
		\subcaption{$H=0.75,$ $T=1$}
	\end{subfigure}

	\caption{Example evolution of fractional Brownian motion under Burgers' equation with the Rusanov flux. \label{fig:fractional_sample_burgers}}
\end{figure}

\subsection{Numerical results}
\Cref{fig:comparison0125,fig:comparison05} show the computed solutions for mesh resolutions of $\Dx=2^{-8}$ and $\Dx=2^{-15}$, and as expected the approximation converges upon mesh refinement. In order to measure the rate of convergence we compare with a reference solution computed on a mesh of $2^{16}$ cells ($\Dx = 2^{-16}$). In order to ensure that the results are representative of a (fractional) Brownian motion, we repeat the experiment for 512 different initial data samples and average the corresponding convergence rates; the results are shown in \Cref{fig:avg_convergence}. \Cref{fig:avg_convergence} clearly show convergence for all the given configurations. However, for most configurations -- most notably those of low Hurst index $H$ -- we observe better convergence rates than those predicted by \Cref{thm:monotone_theorem,thm:maintheorem}.

\subsection{Sharpness of our estimates}
We measure the growth of $\TV(u^\Dx_0)$ as a function of $\Dx$ in \Cref{fig:tv_scaling_dx}. The figure agrees with \eqref{eq:tvprojectionerror} in Lemma \ref{lem:approximation_holder}, and shows that initial data with Hurst index $H$ has a blow up in total variation which scales as $\Dx^{H-1}$. We furthermore measure the scaling of the discrete Lipschitz norm $\DLipNorm{v^0}$ as a function of $\Dx$ in \Cref{fig:lipschitz_scaling_dx}. From the figure it is clear that our estimate of the initial $\text{Lip}^+$ norm has the correct scaling.

In \Cref{fig:tv_scaling_time} we show the evaluation of the inverse of the total variation as a function of time. As we can see from the plot, the total variation decays as $C/t$, which agrees well with the estimate on $\TV(v^n)$ in the proof of \Cref{lem:lipplus}. 

Inspired by \eqref{eq:tv_not_so_sharp}, we use the value of $\beta=\frac{1}{2}\cdot\frac{1}{4}$ for the Godunov flux applied to the Burgers' equation, computed in~\cite{NT92}, and measure the sharpness of the bound~\eqref{eq:tv_not_so_sharp} by the ratio
\[\frac{2M\left(\DLipNorm{v^0}\Dt + \frac{1}{\beta}\log \big(1+\beta t^N\DLipNorm{v^0}\big)\right)}{\sum_{n=0}^N \TV(v^n)\Dt},\]
which is plotted in~\Cref{fig:lip_coefficient_ratio} as a function of the spatial resolution $\Dx$. As we can see, the bound is not perfectly sharp, and seems to overestimate the sum on the left hand side of~\eqref{eq:tv_not_so_sharp} by a factor of $\Delta^{\frac{1}{4}}$. This  partially explains the discrepancy of the observed convergence rates and the predicted convergence rate of~\Cref{thm:maintheorem}.

\subsection{Reproducing the numerical experiments}
All numerical experiments have been done with the Alsvinn code available from \url{https://alsvinn.github.io}. The code for post-processing, along with the generation of the initial data, can be found at \url{https://github.com/kjetil-lye/unbounded_tv_experiments}. They are also permanently stored on the Zenodo platform with the DOI  \texttt{10.5281/zenodo.4088164}.

\section{Conclusion}
In this paper we have extended the standard Kuznetsov convergence proof for finite volume schemes approximating solutions to hyperbolic conservation laws to include initial data with unbounded total variation. The theory covers rapidly oscillating data such as Brownian and fractional Brownian motion. We show several numerical experiments which show good agreement with the theory.

Our result can easily be extended to cover initial data which is only piecewise H\"older continuous with a finite number of downward jump discontinuities. The suboptimal rate $\Dx^{\alpha-1/2}$ in Theorem \ref{thm:monotone_theorem} can be extended even further to cover e.g.~initial data in Besov spaces, since it only relies on the projection estimates \eqref{eq:l1projectionerror} and \eqref{eq:tvprojectionerror}.

We conjecture that the rate $\Dx^{\alpha/2}$ given in \eqref{eq:optimalasymptoticconvrate} in Theorem \ref{thm:maintheorem} is optimal or near optimal. For optimality of convergence rates see e.g.~\cite{Sab97,RuSaSo19}.

\section*{Acknowledgements}
USF was partially funded by the Research Council of Norway project number 301538.

%
%

\begin{figure}
	\begin{subfigure}{0.45\textwidth}
		\InputImage{0.8\textwidth}{0.6\textwidth}{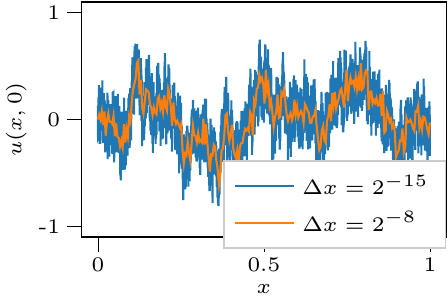}
		\subcaption{Burgers with Godunov, initial}
	\end{subfigure}
	\begin{subfigure}{0.45\textwidth}
		\InputImage{0.8\textwidth}{0.6\textwidth}{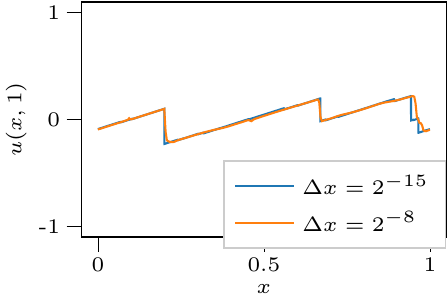}
		\subcaption{Burgers with Godunov, evolved}
	\end{subfigure}

	\begin{subfigure}{0.45\textwidth}
		\InputImage{0.8\textwidth}{0.6\textwidth}{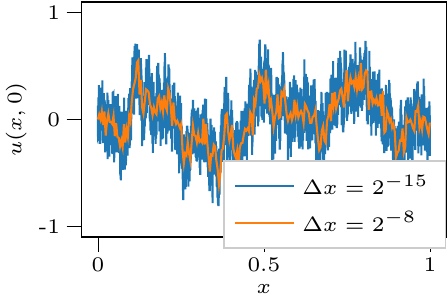}
		\subcaption{Cubic with Rusanov, initial}
	\end{subfigure}
	\begin{subfigure}{0.45\textwidth}
		\InputImage{0.8\textwidth}{0.6\textwidth}{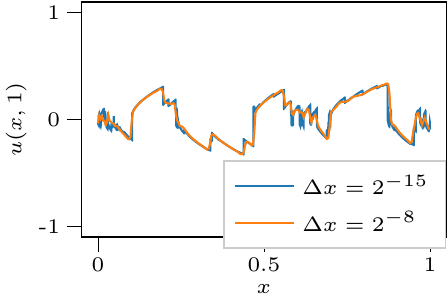}
		\subcaption{Cubic with Rusanov, evolved}
	\end{subfigure}

	\begin{subfigure}{0.45\textwidth}
	\InputImage{0.8\textwidth}{0.6\textwidth}{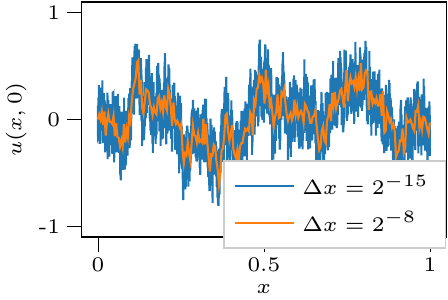}
	\subcaption{Linear with Rusanov, initial}
\end{subfigure}
\begin{subfigure}{0.45\textwidth}
	\InputImage{0.8\textwidth}{0.6\textwidth}{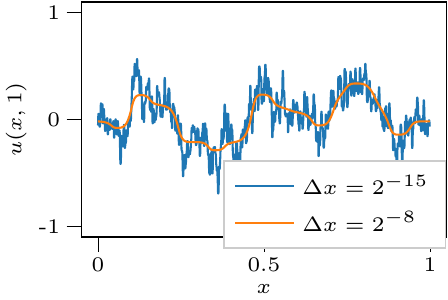}
	\subcaption{Linear with Rusanov, evolved}
\end{subfigure}
\caption{\label{fig:comparison0125}Comparison between different mesh resolutions ($\Dx=2^{-15}$ and $\Dx=2^{-8}$) at $T=0$ and $T=1$ for different equations and schemes. Here $H=0.125$.}
\end{figure}

\begin{figure}
	\begin{subfigure}{0.45\textwidth}
		\InputImage{0.8\textwidth}{0.6\textwidth}{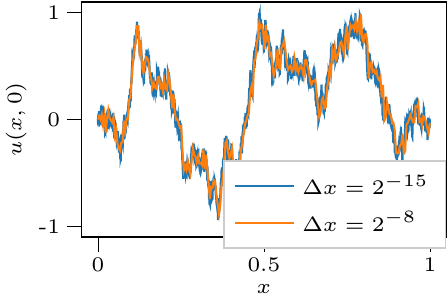}
		\subcaption{Burgers with Godunov, initial}
	\end{subfigure}
	\begin{subfigure}{0.45\textwidth}
		\InputImage{0.8\textwidth}{0.6\textwidth}{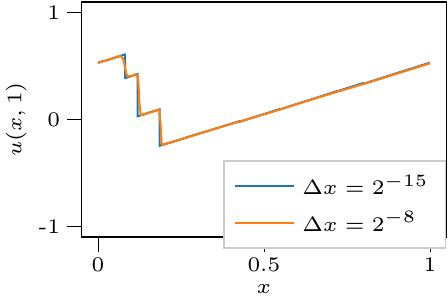}
		\subcaption{Burgers with Godunov, evolved}
	\end{subfigure}

	\begin{subfigure}{0.45\textwidth}
		\InputImage{0.8\textwidth}{0.6\textwidth}{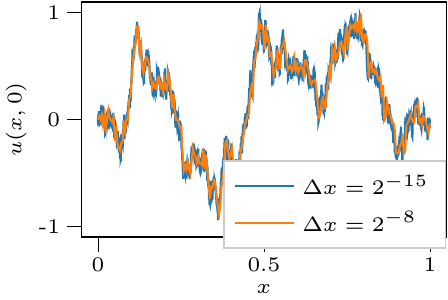}
		\subcaption{Cubic with Rusanov, initial}
	\end{subfigure}
	\begin{subfigure}{0.45\textwidth}
		\InputImage{0.8\textwidth}{0.6\textwidth}{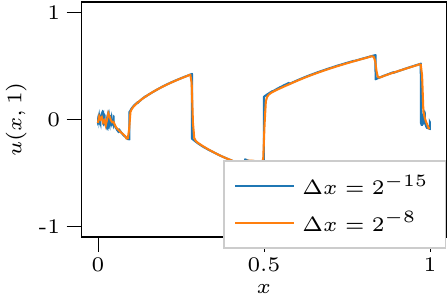}
		\subcaption{Cubic with Rusanov, evolved}
	\end{subfigure}
	
	\begin{subfigure}{0.45\textwidth}
		\InputImage{0.8\textwidth}{0.6\textwidth}{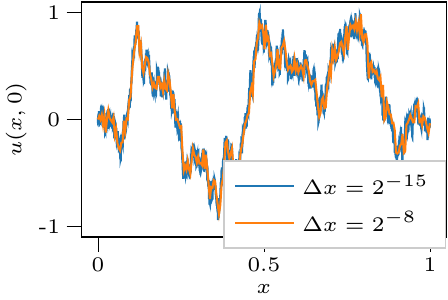}
		\subcaption{Linear with Rusanov, initial}
	\end{subfigure}
	\begin{subfigure}{0.45\textwidth}
		\InputImage{0.8\textwidth}{0.6\textwidth}{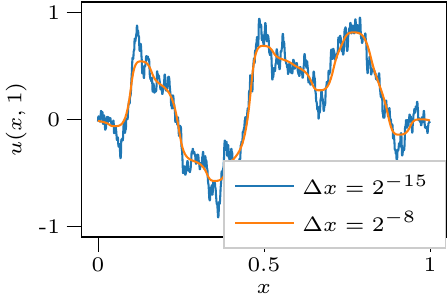}
		\subcaption{Linear with Rusanov, evolved}
	\end{subfigure}
	\caption{\label{fig:comparison05}Comparison between different mesh resolutions ($\Dx=2^{-15}$ and $\Dx=2^{-8}$) at $T=0$ and $T=1$ for different equations and schemes. Here $H=0.5$.}
\end{figure}

\begin{figure}
	\begin{subfigure}{\textwidth}
		\InputImage{0.8\textwidth}{0.4\textwidth}{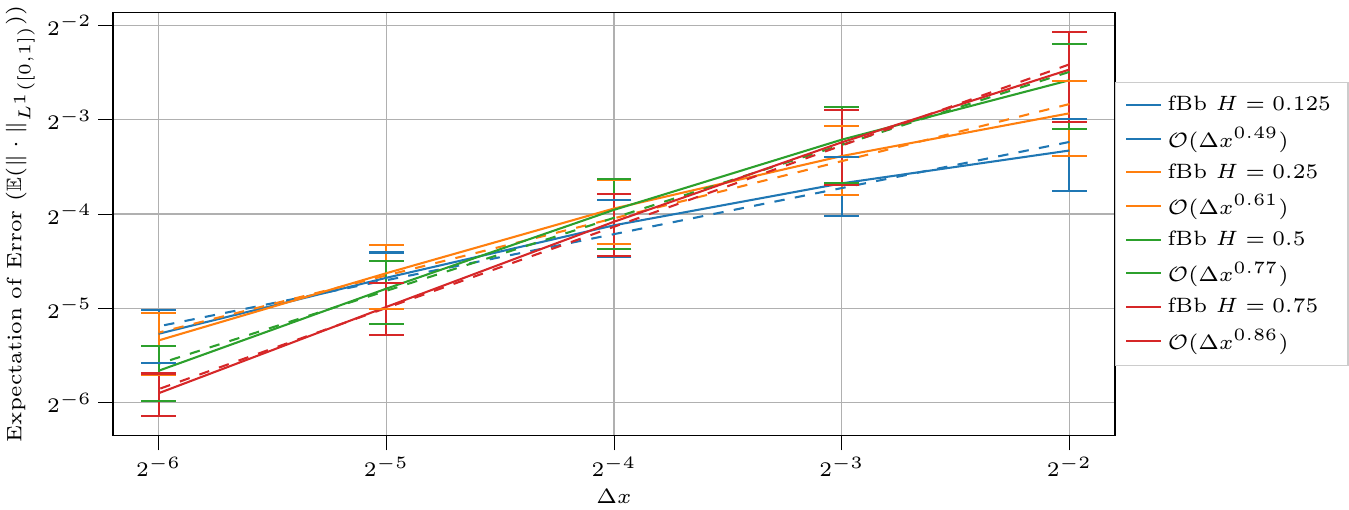}
		\subcaption{Burgers with Godunov}
	\end{subfigure}

	\begin{subfigure}{\textwidth}
		\InputImage{0.8\textwidth}{0.4\textwidth}{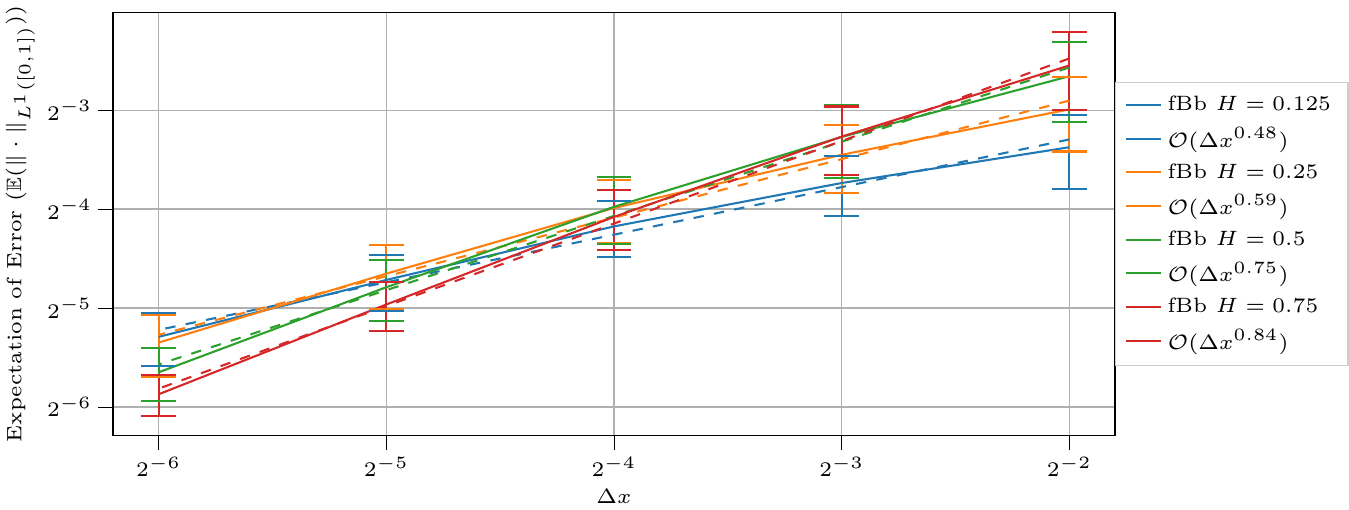}
		\subcaption{Burgers with Rusanov}
	\end{subfigure}
	
	\begin{subfigure}{\textwidth}
		\InputImage{0.8\textwidth}{0.4\textwidth}{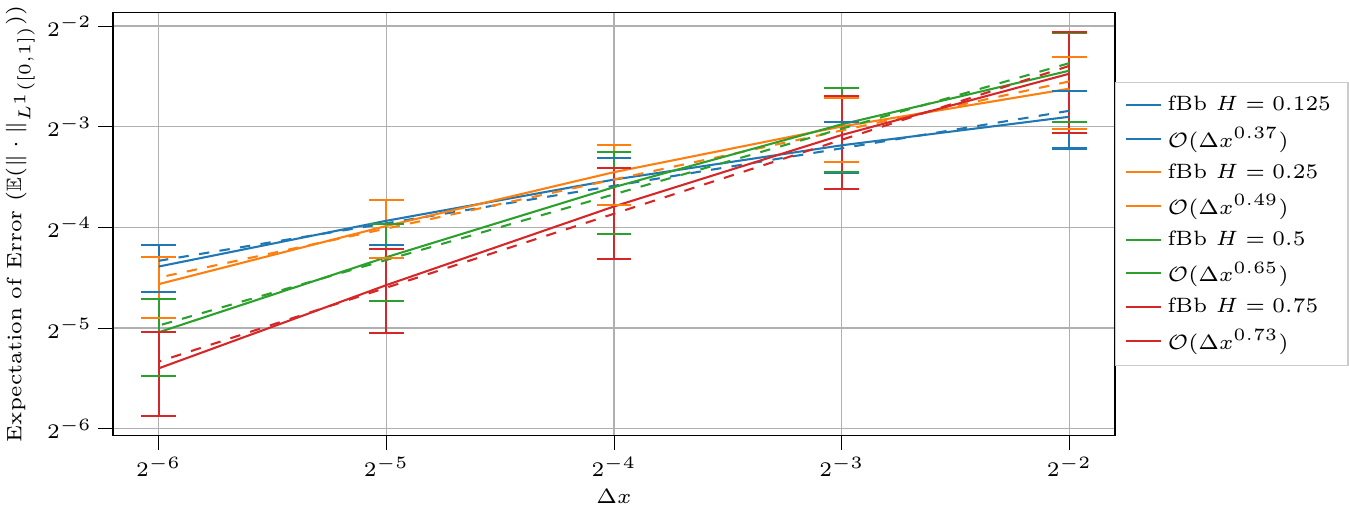}
		\subcaption{Cubic with Rusanov}
	\end{subfigure}

	\begin{subfigure}{\textwidth}
		\InputImage{0.8\textwidth}{0.4\textwidth}{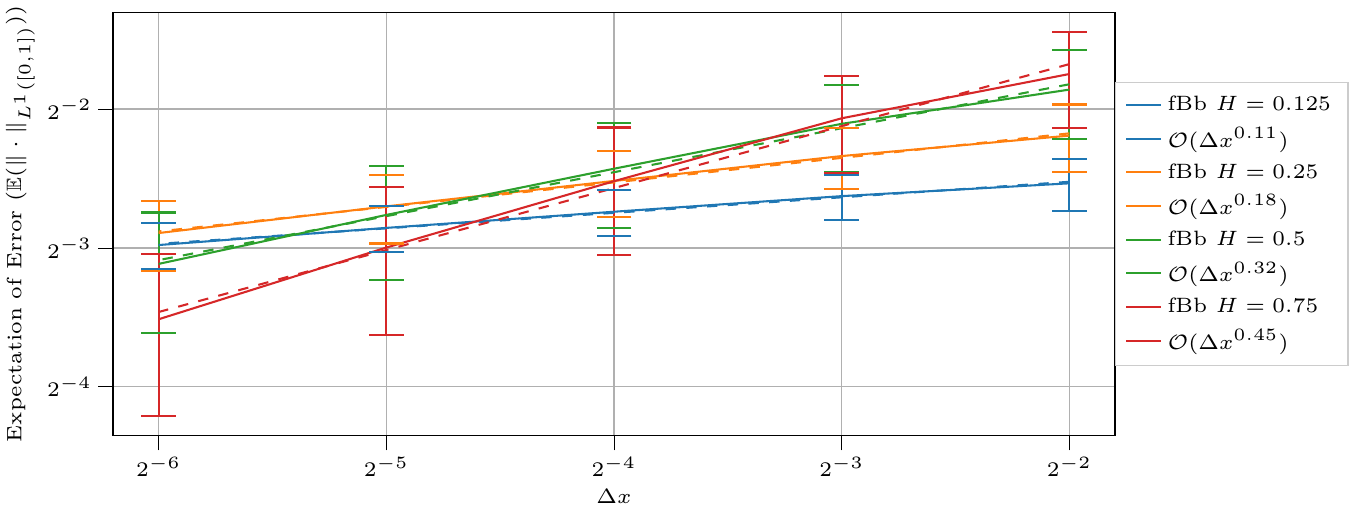}
		\subcaption{Linear with Rusanov}
	\end{subfigure}
	
	\caption{Convergence rates for different equations and numerical fluxes at $T=1$ taken as an average over $512$ samples.\label{fig:avg_convergence}}
\end{figure}

\begin{figure}
	\InputImage{0.8\textwidth}{0.6\textwidth}{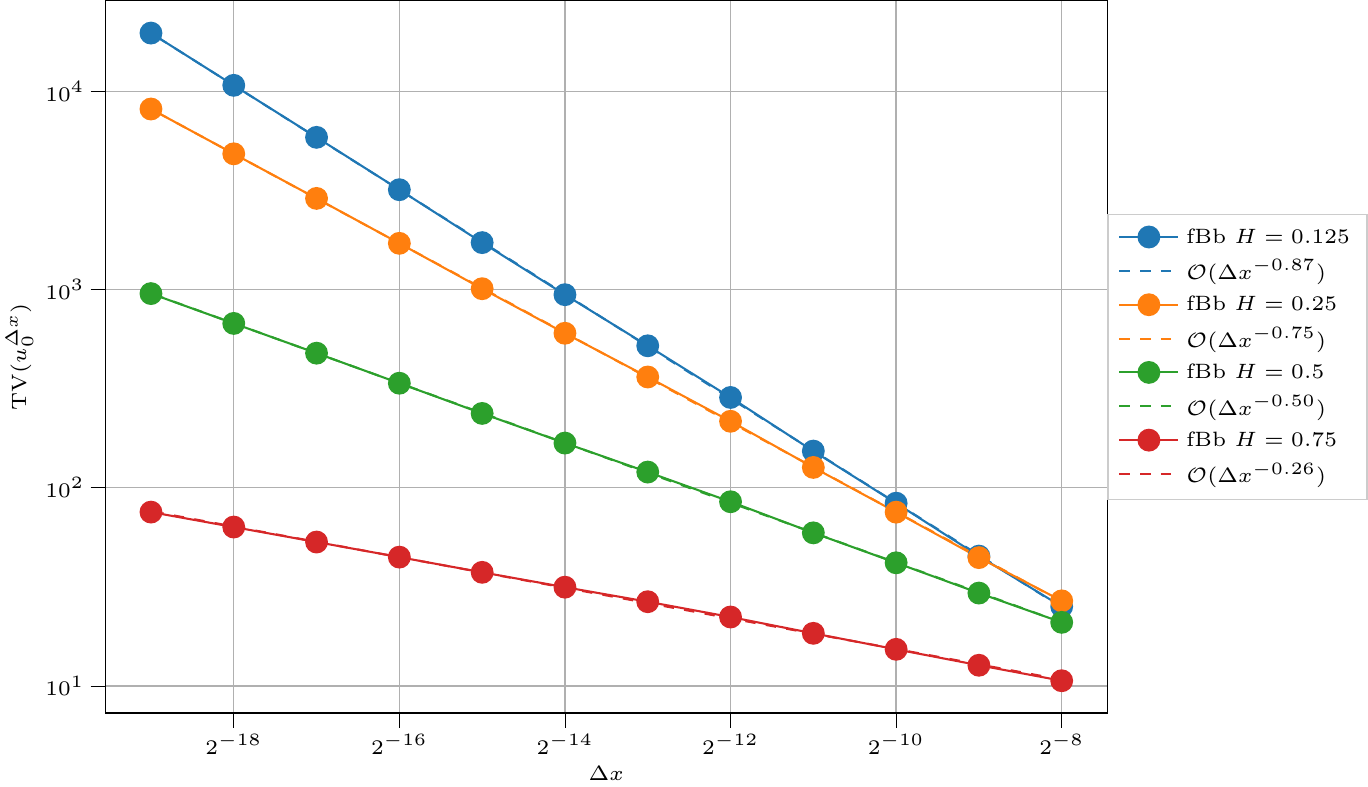}
	\caption{Total variation of initial data as a function of mesh width $\Dx$. (The stippled lines are invisible because they match the solid curves perfectly.) \label{fig:tv_scaling_dx}}
\end{figure}

\begin{figure}
	\InputImage{0.8\textwidth}{0.6\textwidth}{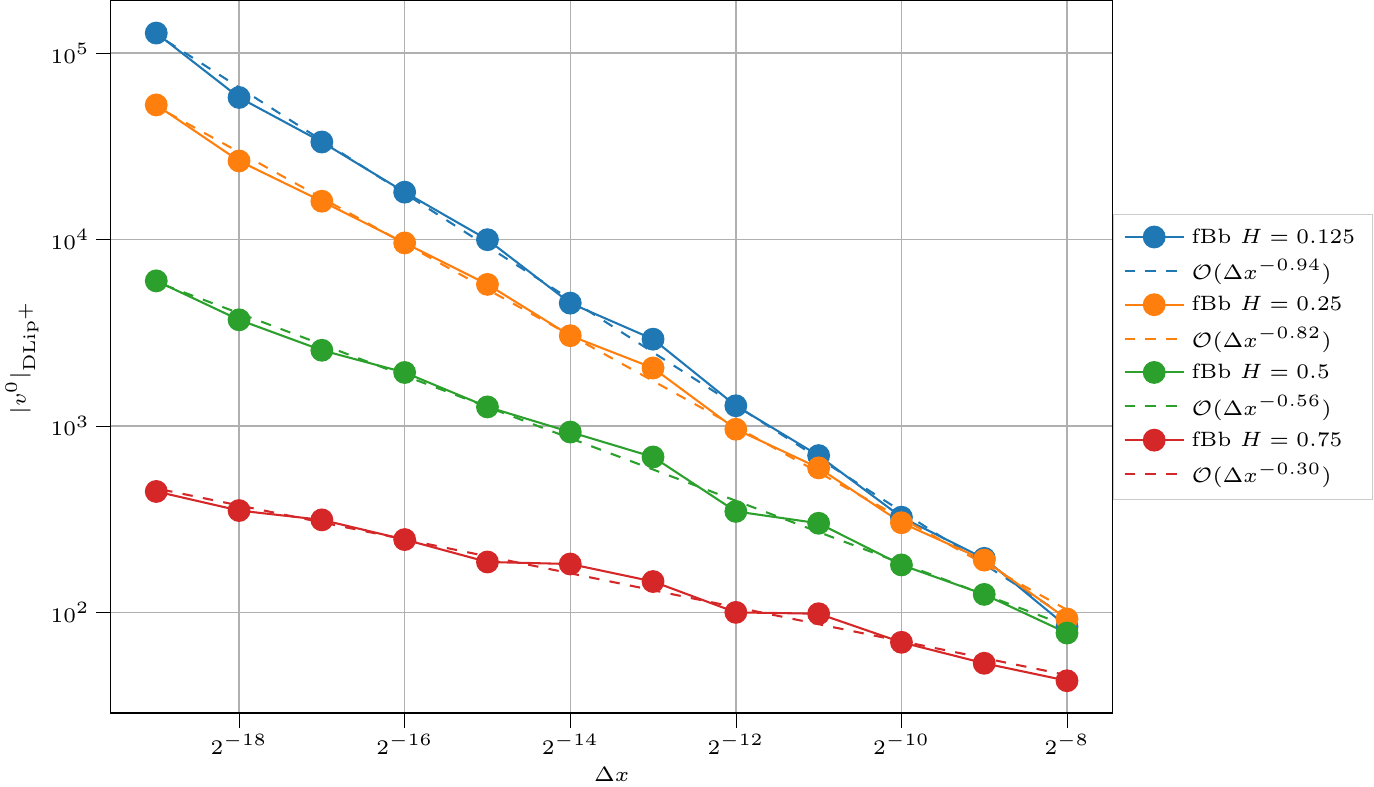}
	\caption{$\DLipNorm{v^0}$ as a function of mesh width $\Dx$. \label{fig:lipschitz_scaling_dx}}
\end{figure}

\begin{figure}
	\begin{subfigure}{0.45\textwidth}
		\InputImage{\textwidth}{0.6\textwidth}{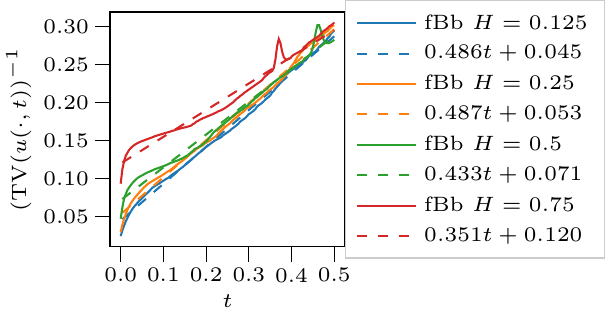}
		\subcaption{Burgers with Godunov}
	\end{subfigure}
	\begin{subfigure}{0.45\textwidth}
		\InputImage{\textwidth}{0.6\textwidth}{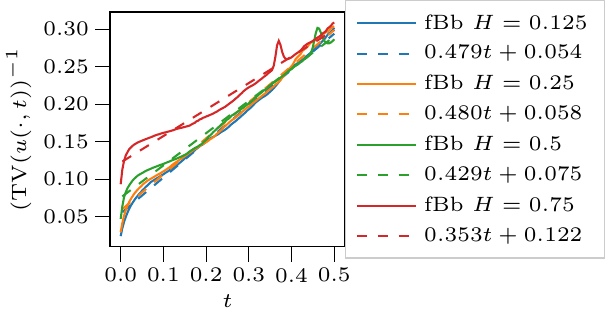}
		\subcaption{Burgers with Rusanov}
	\end{subfigure}

	\begin{subfigure}{0.45\textwidth}
		\InputImage{\textwidth}{0.6\textwidth}{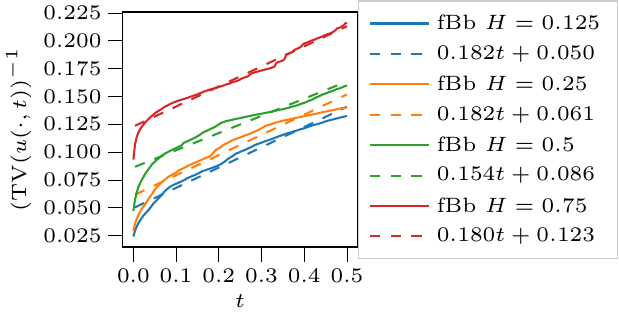}
		\subcaption{Cubic with Rusanov}
	\end{subfigure} \\
	
	\begin{subfigure}{0.45\textwidth}
		\InputImage{\textwidth}{0.6\textwidth}{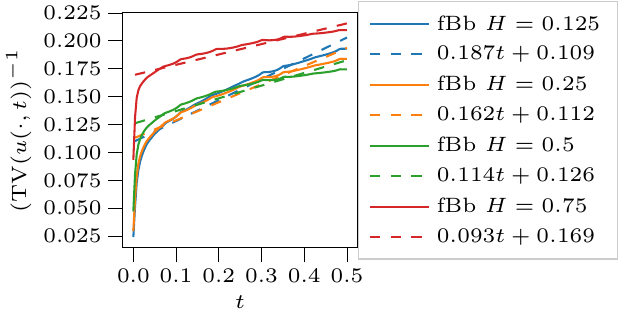}
		\subcaption{Linear with Rusanov}
	\end{subfigure}
	\begin{subfigure}{0.45\textwidth}
		\InputImage{\textwidth}{0.6\textwidth}{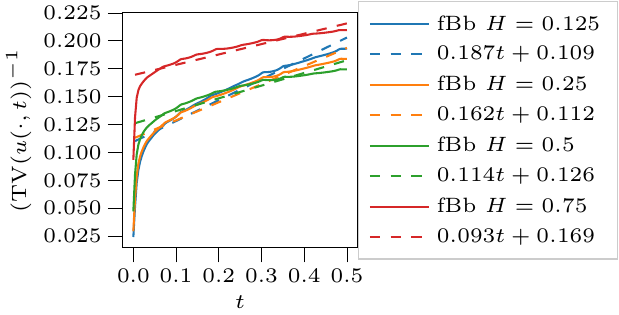}
		\subcaption{Linear with upwinding}
	\end{subfigure}
	\caption{The total variation as a function of time for varying equations, numerical fluxes and Hurst indices. \label{fig:tv_scaling_time}}
\end{figure}

\begin{figure}

	\InputImage{0.9\textwidth}{0.6\textwidth}{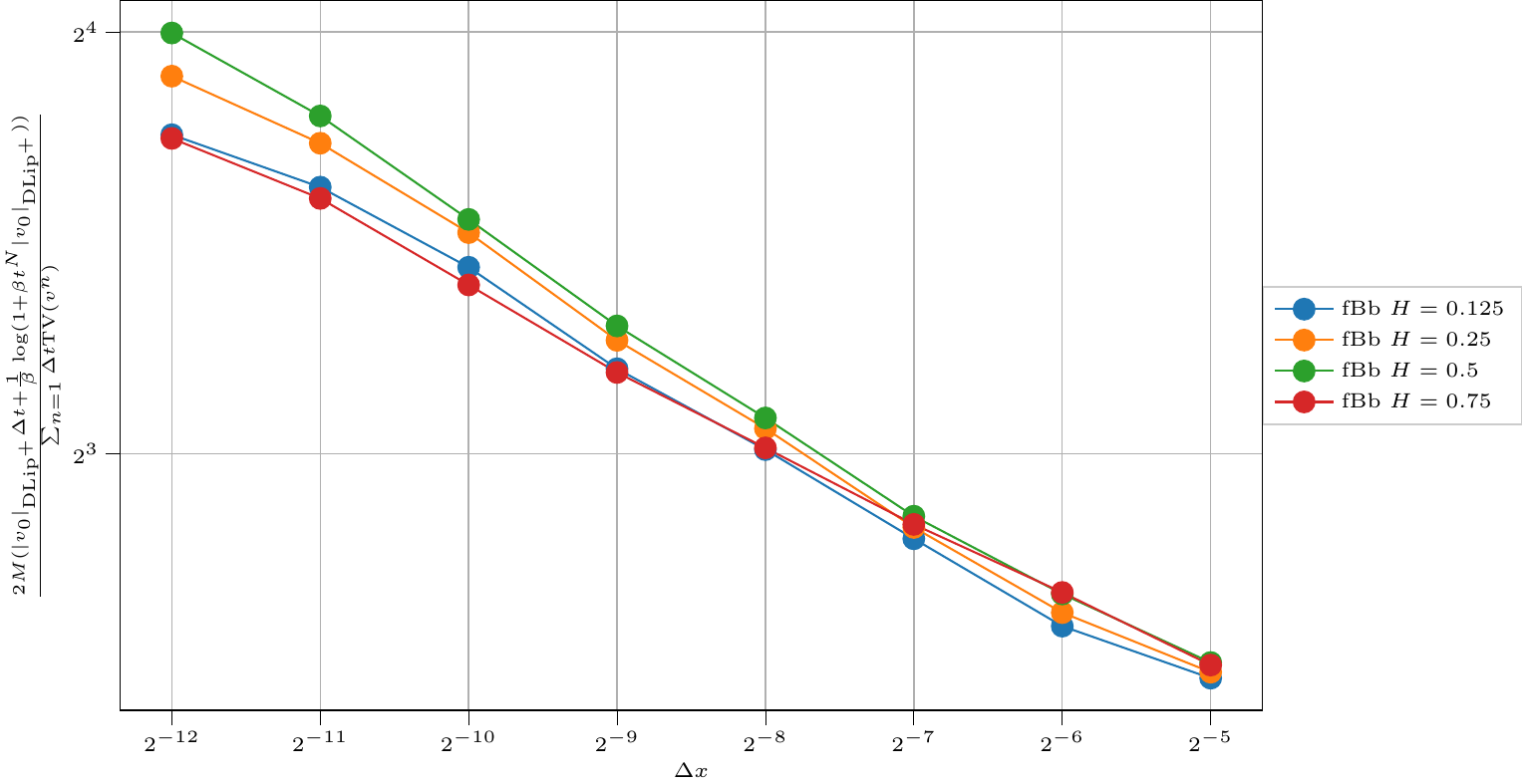}
	

	\caption{\label{fig:lip_coefficient_ratio}Sharpness of the bound \eqref{eq:tv_not_so_sharp} for Burgers' equation solved with the Godunov scheme. Specifically, we plot the ratio $\frac{2M\left(\DLipNorm{v^0}\Dt + \frac{1}{\beta}\log \big(1+\beta t^N\DLipNorm{v^0}\big)\right)}{\sum_{n=0}^N \TV(v^n)\Dt}$ as a function of the mesh width $\Dx$.}
\end{figure}
\clearpage
\bibliographystyle{plain}
\bibliography{biblio}
\end{document}

%% file: header.tex
\usepackage{caption}
\usepackage{subcaption}
\usepackage{pgfplots}
\pgfplotsset{
	compat=newest,
}
\usepgfplotslibrary{groupplots}
\usepackage{amsmath,amssymb,amsfonts, amsthm}
\usepackage{nicefrac}
\usepackage{bbm}
\newcommand{\R}{\mathbb{R}}
\newcommand{\KuznetsovSpace}{\ensuremath{\mathcal{K}}}
\newcommand{\Lone}{\ensuremath{L^1(\R)}}
\newcommand{\Z}{\mathbb{Z}}
\newcommand{\N}{\mathbb{N}}
\newcommand{\Ph}{\R}

\newcommand{\D}{\R}

\newcommand{\dd}{\ensuremath{\,d}}
\newcommand{\Dx}{{\ensuremath{\Delta x}}}

\newcommand{\Dt}{{\ensuremath{\Delta t}}}

\newcommand{\hf}{\ensuremath{\nicefrac{1}{2}}}

\newcommand{\iphf}{{i+\hf}}
\newcommand{\imhf}{{i-\hf}}
\newcommand{\nphf}{{n+\hf}}
\usepackage{subcaption}

\DeclareMathOperator{\Lip}{Lip}
\DeclareMathOperator{\TV}{TV}
\DeclareMathOperator{\BV}{BV}
\DeclareMathOperator{\sign}{sign}
\DeclareMathOperator{\supp}{supp}

\newcommand{\cell}{{\mathcal{C}}}
\newcommand{\LipNorm}[1]{|#1|_{\Lip^+}}
\newcommand{\DLipNorm}[1]{|#1|_{\mathrm{DLip}^+}}
\renewcommand{\phi}{\varphi}
\renewcommand{\leq}{\leqslant}
\renewcommand{\geq}{\geqslant}
\renewcommand{\epsilon}{\varepsilon}

\newcommand{\ind}{\mathbbm{1}}

\usepackage{todonotes}

\newtheorem{theorem}{Theorem}
\newtheorem{lemma}[theorem]{Lemma}
\theoremstyle{definition}

\usepackage{cleveref}

\newlength\figureheight
\newlength\figurewidth

\usepackage{scalerel}[2014/03/10]
\usepackage[usestackEOL]{stackengine}
\def\intavg{\,\ThisStyle{\ensurestackMath{%
  \stackinset{c}{0\LMpt}{c}{0\LMpt}{\SavedStyle-}{\SavedStyle\phantom{\int}}}%
  \setbox0=\hbox{$\SavedStyle\int\,$}\kern-\wd0}\int}

\usepackage{booktabs}
\usepackage{algorithmicx}
\usepackage{algpseudocode}
\usepackage{graphicx}

\usepackage{etoolbox}

\newtoggle{usetikz}
\toggletrue{usetikz}
\togglefalse{usetikz}
\newcommand{\InputImage}[3]{}
\iftoggle{usetikz}{%
	\renewcommand{\InputImage}[3]{%
		\begin{center}%
		\setlength\figureheight{#2}%
		\setlength\figurewidth{#1}%
		\input{experiments/img_tikz/#3_notitle.xyz}%
		\end{center}%
	}%
} { %
	\renewcommand{\InputImage}[3]{%
		\begin{center}%
		\includegraphics[width=#1]{img/#3_notitle}%
		\end{center}%
	}%
}

\pgfplotsset{every tick label/.append style={font=\tiny}}
\pgfplotsset{
	tick label style = {font = {\fontsize{6 pt}{7 pt}\selectfont}},
	label style = {font = {\fontsize{6 pt}{7 pt}\selectfont}},
	legend style = {font = {\fontsize{6 pt}{7 pt}\selectfont}},  
}

\usepackage{scalefnt}
